\documentclass[11pt,reqno]{amsart}
\usepackage{amssymb,latexsym}
\usepackage{graphicx}
\usepackage{color} 
\usepackage{url}		%does nice formatting of URLs
\calclayout

\theoremstyle{plain}
\numberwithin{equation}{section}
\newtheorem{thm}{Theorem}[section]

\newtheorem{lemma}[thm]{Lemma}

\newtheorem{proposition}[thm]{Proposition}
\newtheorem{corollary}[thm]{Corollary}
\newtheorem{remark}[thm]{Remark}

\begin{document}
%% replace the values in the next three lines by the correct information
\setcounter{page}{1}

\title[A combinatorial sum with two complex parameters]
{A combinatorial sum with two complex parameters}
\author{Michel Bataille}
\address{Independent Researcher, 76520 Franqueville-Saint-Pierre, France}
\email{michelbataille@wanadoo.fr}
\medskip
\author{Robert Frontczak}
\address{Independent Researcher, 72764 Reutlingen, Germany}
\email{robert.frontczak@web.de}

\begin{abstract}
This article deals with combinatorial identities with two complex parameters. Starting with a fundamental lemma, we derive various polynomial identities, combinatorial sums and related results. For example, we generalize a polynomial identity of Carlitz involving central binomial coefficients and present a second identity of the same nature. Special cases of our findings lead to sums involving Catalan numbers, harmonic numbers, and Fibonacci numbers. 
\\
\end{abstract}

%\thanks{ }

\maketitle

\section{Introduction}

Identities that express relations between polynomials and combinatorial quantities such as binomial coefficients, falling or rising factorials or other counting numbers are important objects in number theory, combinatorics and probability. Most basic examples are the binomial or multinomial theorem. Identities with a complex parameter in the binomial coefficient are also widely spread in the mathematical literature. 
What comes to mind first is probably the Chu-Vandermonde identity: 
\begin{equation*}
\sum_{k=0}^n \binom{x}{k} \binom{y}{m-k} = \binom{x+y}{n} \binom{y-x}{m-n}
\end{equation*}
where $x$ and $y$ are complex numbers and $m$ and $n$ are nonnegative integers.\\ 

The classical textbooks by Gould \cite{Gould}, Comtet \cite{Comtet}, Riordan \cite{Riordan}, 
and Graham et al. \cite{Graham} contain many additional examples. Such identities still attract attention and are 
the subject of research \cite{Adegoke26,Adegoke,Alzer,Munarini,Boyadzhiev10}. 
For instance, Boyadzhiev \cite[Proposition 2]{Boyadzhiev10} gives the polynomial identity, valid for complex numbers $x$ and $\alpha$,
\begin{equation*}
\sum_{k=0}^n \binom{n}{k} \binom{\alpha+k}{k} x^k = \sum_{k=0}^n \binom{\alpha}{n-k} \binom{\alpha+k}{k} (-1)^{n-k} (1+x)^k,
\end{equation*}
and shows that this identity contains the polynomial identities of Simons and Ljunggren as special cases. \\

In his book \cite{Gould}, Henry Gould lists identities valid for a polynomial $f(x)$ (identities numbered Z1 to Z12). 
Let $f(x)$ be any polynomial in $x$ of degree $\leq n$. Then identity Z5 is
\begin{equation}\label{Gould_id}
f(x+y) = y \binom{y+n}{n} \sum_{k=0}^n \binom{n}{k} (-1)^k \frac{f(x-k)}{y+k}.
\end{equation}
If we set $y=x$ and $f(x)=1$ then
we obtain the well-known decomposition
\begin{equation}\label{partial}
\sum_{k=0}^n \binom{n}{k} \frac{(-1)^k}{x+k} = \frac{1}{x\binom{x+n}{n}} \qquad x\notin \{-n,-(n-1),\ldots,0\}.
\end{equation}
In this paper, we work with an identity similar in type to \eqref{Gould_id}. It involves two complex parameters and appears very productive. This is first shown by several applications leading to a number of identities and evaluations of combinatorial sums (Catalan's numbers intervene in some of them). Then the focus turns to interesting results about harmonic numbers. Finally, a special family of combinatorial sums is considered and studied in details. To conclude the paper, a few examples of applications to Fibonacci numbers are offered.\\[0.3cm]
For the convenience of the reader, we recall the following definition: 
The harmonic numbers $H_z$ are defined by $H_0=0$ and for $0\ne z\in\mathbb C\setminus\mathbb Z^{-}$ by the relation
\begin{equation*}
H_z = H_{z - 1} + \frac{1}{z}
\end{equation*}
It is worth noticing that
\begin{equation}\label{Har_psi}
H_z = \psi(z + 1) + \gamma,\qquad (z\notin \mathbb{Z}^-)
\end{equation}
where $\gamma$ is the Euler-Mascheroni constant and $\psi(z)=\Gamma'(z)/\Gamma(z)$ is the digamma (or psi) function. 
Of course, when $z=n$, a positive integer, then $H_n=\sum_{k=1}^n \frac{1}{k}$.

\section{A fundamental lemma and first consequences}

The results of this study are derived from the following fundamental lemma.

\begin{lemma}
Let $n$ be a nonnegative integer and let $x$ and $z$ be two complex numbers such that $z\notin\{0,1,\ldots,n\}$. Then we have
\begin{equation}\label{main_lem_id}
\sum_{k=0}^n \binom{z}{k} (-1)^k (x+1)^k = (-1)^n (n+1)\binom{z}{n+1} \sum_{k=0}^n \binom{n}{k} \frac{x^k}{z-k}.
\end{equation}
\end{lemma}
\begin{proof}
We use the following known relations where $z$ is a complex number 
\[\binom{z}{s}\binom{s}{t}=\binom{z}{t}\binom{z-t}{s-t}\quad (s,t \ \mbox{integers and } \ s\ge t\ge 0)\]
and
\[\sum_{k=0}^p(-1)^k\binom{z}{k}=(-1)^p\binom{z-1}{p}\quad (p\ \mbox{integer and }\ p\ge 0).\]
We calculate
\begin{align*}
\sum_{k=0}^n \binom{z}{k} (-1)^k (x+1)^k &= \sum_{k=0}^n \binom{z}{k} (-1)^k \sum_{j=0}^k \binom{k}{j} x^j \\
&= \sum_{j=0}^n \sum_{k=j}^n \binom{z}{j} \binom{z-j}{k-j} (-1)^k x^j \\
&= \sum_{j=0}^n \binom{z}{j} x^j \sum_{r=0}^{n-j} \binom{z-j}{r} (-1)^{r+j} \\
&= \sum_{j=0}^n \binom{z}{j} (-1)^j x^j (-1)^{n-j} \binom{z-j-1}{n-j} \\
&= (-1)^n \sum_{j=0}^n \binom{z}{j} \binom{z-j-1}{n-j} x^j
\end{align*}
and \eqref{main_lem_id} follows because
\[\binom{z}{j} \binom{z-j-1}{n-j} = \frac{z(z-1)\cdots(z-n)}{z-j}\cdot\frac{1}{j!(n-j)!} = (n+1)\binom{z}{n+1}\binom{n}{j}\frac{1}{z-j}.\]
\end{proof}

Identity \eqref{main_lem_id} is very powerful, as will be shown by the number of interesting combinatorial results obtained in what follows.
As a ``warm up", setting $z=-1$ and $x=-1$ in turn yield the elementary identity
$$\sum_{k=0}^n (x+1)^k = \sum_{k=0}^n \binom{n+1}{k+1} x^k$$
and the well-known decomposition into partial fractions as stated in \eqref{partial} (see also \cite{Boyadzhiev})
\begin{equation}\label{eq_9} 
\frac{ n!}{z(z-1)\cdots (z-n)} = \sum_{k=0}^n \binom{n}{k} \frac{(-1)^{n-k}}{z-k}.
\end{equation}
In a way, identity \eqref{main_lem_id} can be considered as generalizing either of these results.\\[0.3cm]

In general, from \eqref{main_lem_id} (or a modified version), we will deduce identities with only one complex parameter and obtain closed forms of combinatorial sums (or identities about them). To illustrate this scheme, here is an introductory example.

\begin{proposition}
Let $n$ be a nonnegative integer and let $x$ and $z$ be two complex numbers such that 
$z\notin\{-1,-2,\ldots,-(n+1)\}$. Then we have
\begin{equation}\label{main_2}
\sum_{k=0}^n \binom{z+k}{k} (x+1)^k = (n+1)\binom{z+n+1}{n+1} \sum_{k=0}^n \binom{n}{k} \frac{x^k}{z+k+1}.
\end{equation}
\end{proposition}
\begin{proof}
We apply \eqref{main_lem_id} with $-(z+1)$ instead of $z$ and use $\binom{-(z+1)}{k}=(-1)^k\binom{z+k}{k}$.
\end{proof}

The following polynomial identity in the variable $z$ is immediately obtained by taking $x=0$:

\begin{corollary}\label{main_2_cor_1 }
If $z$ is a complex number, then for any integer $n\ge 0$
\begin{equation}
\sum_{k=0}^n \binom{z+k}{k} = \binom{z+n+1}{n}.
\end{equation}
In particular, for any nonnegative integer $q$, we have
\begin{equation*}
\sum_{k=0}^n \binom{q+k}{k} = \binom{q+n+1}{n},
\end{equation*}
the well-known hockey-stick formula (see \cite{Lynch} for a different generalization).
\end{corollary}

The next corollary evaluates a combinatorial sum.

\begin{corollary} 
If $n$ is a nonnegative integer, then the following identity holds:
\[\sum_{k=0}^n \binom{n}{k} \frac{(-1)^k}{2^k(n+k+1)} = \frac{2^n}{(2n+1)\binom{2n}{n}}.\]
\end{corollary}
\begin{proof}
In \eqref{main_2}, take $z=n,x=-\frac{1}{2}$ and recall the well-known $\sum_{k=0}^n \binom{n+k}{k}\frac{1}{2^k}=2^n$.
\end{proof}

\section{First applications}

We now apply the scheme described and illustrated above to obtain more elaborate identities from \eqref{main_lem_id}. 
Each proposition of this section is the starting point of a new scheme.\\

Our first example is directly obtained by replacing $z$ by $z+1/2$:
\begin{proposition}
Let $x$ and $z$ be two complex numbers with $z\notin\{-1/2,1/2,\ldots,n-1/2\}$. Then we have
\begin{equation}\label{id_C1}
\sum_{k=0}^n \binom{z+1/2}{k} (-1)^k (1+x)^k= (-1)^n 2(n+1) \binom{z+1/2}{n+1}\sum_{k=0}^n \binom{n}{k} \frac{x^k}{2z-2k+1}.
\end{equation}
\end{proposition}

As a corollary we get two polynomial identities.

\begin{corollary}\label{polynomial1}
If $x$ is a complex number, then for any integer $n\ge 0$ we have
\begin{equation}\label{id 1}
\sum_{k=0}^n \binom{2k}{k} 2^{-2k} x^k = 2^{-2n} (2n+1) \binom{2n}{n} \sum_{k=0}^n \binom{n}{k} (-1)^k \frac{(1-x)^k}{2k+1}
\end{equation}
and
\begin{equation}\label{cor_C11}
\sum_{k=0}^n \binom{2k}{k} 2^{-2k} \frac{(1+x)^k}{2k-1} = 2^{-2n} \binom{2n}{n} \sum_{k=0}^n \binom{n}{k} \frac{x^k}{2k-1}.
\end{equation}
\end{corollary}
\begin{proof}
To obtain \eqref{id 1}, we take $z=-1$ in \eqref{id_C1}, use 
$$\binom{-1/2}{k}=(-1)^k 2^{-2k}\binom{2k}{k}$$ 
and then replace $x$ by $x-1$. \\
Second, we take $z=0$ in \eqref{id_C1} and use 
$$\binom{1/2}{k}=(-1)^{k-1}2^{-2k}\frac{1}{2k-1}\binom{2k}{k}.$$ 
In both cases, we simplify with the help of
$$\binom{2(n+1)}{n+1} = \frac{2(2n+1)}{n+1} \binom{2n}{n}.$$
\end{proof}

\begin{remark}
Identity \eqref{id 1} is known as Carlitz's polynomial relation \cite{Carlitz}. Identity \eqref{cor_C11} seems to be new.
\end{remark}

Two known examples from the previous corollary are
$$\sum_{k=0}^n \binom{2k}{k} 2^{-2k} \frac{1}{2k-1} = -2^{-2n} \binom{2n}{n},$$
and
$$\sum_{k=0}^n \binom{n}{k} \frac{(-1)^k}{2k-1} = - \frac{2^{2n}}{\binom{2n}{n}}.$$
The first example appears in Riordan's book \cite[p. 130]{Riordan} and also as equation (31) in \cite{Adegoke}. \\

In a similar way, we obtain

\begin{proposition}
Let $x$ and $z$ be two complex numbers with $z\notin\{1/2,3/2,\ldots,n+1/2\}$. Then we have
\begin{equation}\label{id_C2}
\sum_{k=0}^n \binom{z-1/2}{k}(-1)^k(x+1)^k= (-1)^n 2(n+1) \binom{z-1/2}{n+1}\sum_{k=0}^n \binom{n}{k} \frac{x^k}{2z-2k-1}.
\end{equation}
\end{proposition}

\begin{corollary}\label{polynomial2}
For each complex $x$ we have
\begin{equation}\label{cor_C21}
\sum_{k=0}^n \frac{\binom{n}{k}}{\binom{2k}{k}} 2^{2k} (-1)^{k+1} (1+x)^{n-k} = \sum_{k=0}^n \binom{n}{k} \frac{x^{n-k}}{2k-1}.
\end{equation}
\end{corollary}
\begin{proof}
Set $z=n$ in \eqref{id_C2} and use the easily checked formulas:
$$\binom{n-1/2}{k}=2^{-2k}\frac{\binom{2n}{n}\binom{n}{k}}{\binom{2n-2k}{n-k}}\qquad (k=0,1,\ldots,n)$$
and
$$\binom{n-1/2}{n+1}=-\frac{2^{-2n-1}}{n+1}\binom{2n}{n}.$$
\end{proof}

As an explicit evaluation we offer the sum
$$\sum_{k=0}^n \frac{\binom{n}{k}}{\binom{2k}{k}} 2^{2k} (-1)^{k+1} = \frac{1}{2n-1}.$$
This sum is also given by equation (32) in \cite{Adegoke}.\\ 

The polynomial identities stated in Corollaries \ref{polynomial1} and \ref{polynomial2} can be merged in two
different ways, hereby leading to additional interesting relations and evaluations.

\begin{corollary}
For each complex $x\neq -1$ we have the following relation
\begin{equation}\label{id_C3}
\sum_{k=0}^n \binom{2k}{k} \frac{2^{-2k}}{2k-1} \left (\frac{1}{1+x}\right )^{n-k}
= 2^{-2n} \binom{2n}{n} \sum_{k=0}^n \frac{\binom{n}{k}}{\binom{2k}{k}} 2^{2k} (-1)^{k+1}\left ( \frac{x}{1+x}\right )^{k}.
\end{equation}
In particular,
\begin{equation*}
\sum_{k=0}^n \binom{2k}{k} \frac{2^{-k}}{2k-1} =2^{-n}\binom{2n}{n}\sum_{k=0}^n\frac{\binom{n}{k}}{\binom{2k}{k}}(-1)^{k+1} 2^{k}
\end{equation*}
and
\begin{equation*}
\sum_{k=0}^n \binom{2k}{k} \frac{2^{-3k}}{2k-1} = - 2^{-3n} \binom{2n}{n} \sum_{k=0}^n \frac{\binom{n}{k}}{\binom{2k}{k}}2^{2k}.
\end{equation*}
\end{corollary}
\begin{proof}
From \eqref{cor_C11} we deduce that
\[ \sum_{k=0}^n \binom{2k}{k}\frac{2^{-2k}}{2k-1}(1+x)^{k-n}=2^{-2n}\binom{2n}{n}(1+x)^{-n}x^n\sum_{k=0}^n\binom{n}{k}\frac{1}{2k-1}\frac{1}{x^{n-k}}.\]
Applying \eqref{cor_C21} (with $\frac{1}{x}$ instead of $x$) gives the conclusion. The special cases follow upon setting first $x=1$ and then $x=-1/2$ in \eqref{id_C3}.
\end{proof}

We advantageously rewrite \eqref{cor_C11} and \eqref{cor_C21} as follows:

\begin{corollary}
\begin{equation}\label{id_Ca}
\sum_{k=0}^n\frac{\binom{n}{k}}{\binom{2k}{k}}2^{2k}(-1)^{k+1} x^k (1+x)^{n-k} = \sum_{k=0}^n\binom{n}{k}\frac{x^k}{2k-1}
=\frac{2^{2n}}{\binom{2n}{n}}\sum_{k=0}^n\binom{2k}{k} 2^{-2k}\frac{(1+x)^{k}}{2k-1}.
\end{equation}
In particular, we get the following identities involving the Catalan numbers $C_n=\frac{1}{n+1}\binom{2n}{n}$
\begin{equation}\label{id_Cb}
\sum_{k=0}^n \frac{2^{2k}}{\binom{2k}{k}} = \frac{1}{3} \left(1+\frac{2^{2n+1}}{C_n} \right)
\end{equation}
and for $n\ge 1$,
\begin{equation}\label{id_Cc}
\sum_{k=1}^n C_k \frac{2^{-2k}}{2k-1} = \frac{1}{3}\left( 1 - 2^{-2n} C_n \right).
\end{equation}
\end{corollary}
\begin{proof}
In \eqref{cor_C21}, replace $x$ by $\frac{1}{x}$ and multiply by $x^n$. This gives the main statement. In \eqref{id_Ca}, we replace $x$ by $-x$ and integrate from $0$ to $1$ to obtain:
\begin{equation*}%\label{id_Cd}
-\sum_{k=0}^n \frac{\binom{n}{k}2^{2k}}{\binom{2k}{k}}\frac{1}{(n+1)\binom{n}{k}}
= \sum_{k=0}^n\binom{n}{k}\frac{(-1)^k}{(k+1)(2k-1)} = \frac{2^{2n}}{\binom{2n}{n}}\sum_{k=0}^n\frac{\binom{2k}{k}}{2^{2k}(k+1)(2k-1)}.
\end{equation*}
From 
$$\frac{1}{(k+1)(2k-1)} = \frac{2}{3}\cdot\frac{1}{2k-1} - \frac{1}{3}\cdot\frac{1}{k+1},$$ 
we then deduce that
\begin{align*}
\sum_{k=0}^n\binom{n}{k}\frac{(-1)^k}{(k+1)(2k-1)} &= \frac{2}{3}\sum_{k=0}^n\binom{n}{k}\frac{(-1)^k}{2k-1}-\frac{1}{3}\sum_{k=0}^n\binom{n}{k}\frac{(-1)^k}{k+1} \\
&= -\frac{2}{3}\frac{2^{2n}}{\binom{2n}{n}}-\frac{1}{3(n+1)}
\end{align*}
(on the right, the first sum has been met earlier and the second sum follows from the obvious $\sum_{k=0}^n\binom{n}{k}\frac{x^{k+1}}{k+1}=\frac{(x+1)^{n+1}-1}{n+1}$).
\end{proof}

\begin{remark}
Identity \eqref{id_Cb} is a classical result that appeared in a range of papers. It was generalized in a different way in \cite{Bataille2}.
\end{remark}

We now modify \eqref{main_lem_id} \textit{via} integration relatively to $x$: 

\begin{proposition}
Let $n$ be a nonnegative integer and let $x$ and $z$ be two complex numbers such that $z\notin\{0,1,\ldots,n\}$. Then we have
\begin{equation}\label{id_Cz}
\sum_{k=0}^n\binom{z}{k}\frac{(-1)^k}{k+1}(1+x)^{k+1}=\sum_{k=0}^n\binom{z}{k}\frac{(-1)^k}{k+1}+(-1)^n(n+1)\binom{z}{n+1}\sum_{k=0}^n\binom{n}{k}\frac{x^{k+1}}{(k+1)(z-k)}
\end{equation}
\end{proposition}
\begin{proof}
In \eqref{main_lem_id}, consider $x$ as a real variable and integrate from $0$ to the  real $X$, to obtain a polynomial identity in $X$; then, rename $X$ as $x$. 
\end{proof}

\begin{corollary}
If $z$ is a complex number with $z\neq -1$, then we have
\begin{equation}\label{id_Cx}
\sum_{k=0}^n \binom{z}{k} \frac{(-1)^k}{k+1} = \frac{1}{z+1}\left (1+(-1)^n \binom{z}{n+1}\right ).
\end{equation}
In addition, if $z\notin\{-1,0,1,\ldots,n\}$, then 
\begin{equation}\label{id_Cy}
\sum_{k=0}^{\lfloor{n/2}\rfloor} \binom{z}{2k} \frac{1}{2k+1} = \frac{1}{z+1} \binom{z}{n+1} \left(\frac{(-1)^{n}+1}{2} + (-1)^n (n+1) 
\sum_{k=0}^n \binom{n}{k} \frac{(-1)^k 2^k}{z-k} \right ).
\end{equation}
\end{corollary}
\begin{proof}
Taking successively $x=-1$ and $x=-2$ in \eqref{id_Cz} gives 
\begin{equation*}
\sum_{k=0}^n\binom{z}{k}\frac{(-1)^k}{k+1} = (-1)^n(n+1)\binom{z}{n+1}\sum_{k=0}^n\binom{n}{k}\frac{(-1)^k}{(k+1)(z-k)}
\end{equation*}
and
\begin{equation*}
\sum_{k=0}^{\lfloor{n/2}\rfloor}\binom{z}{2k}\frac{1}{2k+1} = (-1)^n(n+1)\binom{z}{n+1}\sum_{k=0}^n\binom{n}{k}\frac{(-1)^k 2^k}{(k+1)(z-k)}.
\end{equation*}
We simplify further according to
$$\frac{1}{(z-k)(k+1)} = \frac{1}{z+1}\left (\frac{1}{z-k}+\frac{1}{k+1} \right ).$$
Identity \eqref{id_Cx} follows as we can calculate
\begin{align*}
\sum_{k=0}^n \binom{z}{k} \frac{(-1)^k}{k+1} &= \frac{(-1)^n (n+1)}{z+1}\binom{z}{n+1}\left ( \sum_{k=0}^n \binom{n}{k} \frac{(-1)^k}{z-k} + 
\sum_{k=0}^n \binom{n}{k} \frac{(-1)^k}{k+1} \right ) \\
&= \frac{(-1)^n (n+1)}{z+1}\binom{z}{n+1}\left ( \frac{(-1)^n}{(n+1)\binom{z}{n+1}} + \frac{1}{n+1} \right ),
\end{align*}
where we have also used \eqref{eq_9}. Identity \eqref{id_Cy} is similarly obtained. 
\end{proof}

For later use note that the replacement of $z$ with $-(z+1)$ in \eqref{id_Cx} shows that it can be equivalently written as 
\begin{equation}\label{eq_8}
\sum_{k=0}^n \binom{z+k}{k} \frac{1}{k+1} = \frac{1}{z} \left(\binom{z+n+1}{n+1} - 1 \right)
\end{equation}
for $z\neq 0$.

\begin{corollary}
The following evaluations hold
\begin{equation}\label{Sn1}
\sum_{k=0}^n \binom{n}{k} \frac{(-1)^k}{(k+1)(n-k+1)} = \frac{1+(-1)^n}{(n+1)(n+2)}
\end{equation}
and
\begin{equation}
\sum_{k=0}^{\lfloor{n/2}\rfloor} \binom{n+1}{2k} \frac{1}{2k+1} = \frac{1}{n+2}\left (2^{n+1} + \frac{(-1)^n-1}{2}\right ).
\end{equation}
\end{corollary} 
\begin{proof}
Both sums are particular cases of \eqref{id_Cx} and \eqref{id_Cy} for $z=n+1$. For the second sum we also used the elementary evaluation
$$\sum_{k=0}^n \binom{n}{k} \frac{(-1)^k 2^k}{n+1-k} = (-1)^n \frac{2^{n+1}-1}{n+1}.$$
\end{proof}

The sum \eqref{Sn1} will be encountered gain in Section 5.\\

We proceed by considering special cases of \eqref{id_Cx}.

\begin{corollary}
Let $z$ be a complex number such that $z\neq -3/2$. Then we have
\begin{equation}\label{from_Cx1}
\sum_{k=0}^n \binom{z+1/2}{k}\frac{(-1)^k}{k+1} = \frac{2}{2z+3}\left ( 1 + (-1)^n \binom{z+1/2}{n+1} \right ).
\end{equation}
Also, if $z$ is a complex number such that $z\neq -1/2$, then
\begin{equation}\label{from_Cx2}
\sum_{k=0}^n \binom{z-1/2}{k}\frac{(-1)^k}{k+1}= \frac{2}{2z+1}\left ( 1 + (-1)^n \binom{z-1/2}{n+1} \right ).
\end{equation}
In particular, we have
\begin{equation}
\sum_{k=0}^n 2^{-2k}C_k=2-2^{-2n-1}(2n+1)C_n
\end{equation}
\begin{equation}
\sum_{k=0}^n C_k \frac{2^{-2k}}{2k-1} = - \frac{1}{3}\left ( 2 + 2^{-2n} C_n \right )
\end{equation}
and
\begin{equation}
\sum_{k=0}^n \frac{\binom{n}{k}}{\binom{2(n-k)}{n-k}} 2^{-2k} \frac{(-1)^k}{k+1} = \frac{1}{(n+1)(2n+1)}\left ( \frac{2}{C_n} - (-1)^n 2^{-2n}\right ).
\end{equation}
\end{corollary}
\begin{proof}
The two first identities are obtained by taking $z=-1$ and $z=0$ in \eqref{from_Cx1} and using results already seen in the proof of corollary 3.2. The third identity is derived by setting $z=n$ in \eqref{from_Cx2} and applying the results used in the proof of corollary 3.5.
\end{proof}

\section{Consequences involving harmonic numbers}

This section presents applications to relations involving the harmonic numbers.

\begin{proposition}
Let $z$ be a complex number such that the involved harmonic numbers are defined. Then we have
\begin{equation}\label{Har_from_Cx}
\sum_{k=0}^n \binom{z}{k} \frac{(-1)^k}{k+1} H_{z-k} 
= \frac{H_z}{1+z} + \frac{1}{(1+z)^2} + \frac{(-1)^n}{1+z} \binom{z}{n+1} \left ( H_{z-n-1} + \frac{1}{1+z} \right ).
\end{equation}
\end{proposition}
\begin{proof}
Differentiate \eqref{id_Cx} taking into account that
$$\frac{d}{dz} \binom{z}{k} = \binom{z}{k} (H_z - H_{z-k}).$$
When simplifying use \eqref{id_Cx} again.
\end{proof}

\begin{corollary}
The following evaluations hold
\begin{equation}\label{Hxxx_id1}
\sum_{k=0}^n \binom{n}{k} \frac{(-1)^k}{k+1} H_{n-k} = \frac{1}{n+1}\left ( H_n + \frac{1-(-1)^n}{n+1}\right ),
\end{equation}
\begin{equation}\label{Hxxx_id2}
\sum_{k=0}^n \binom{n}{k} \frac{(-1)^k}{(k+1)(n-k+1)} H_{n+1-k} = \frac{1}{(n+1)(n+2)}\left ( H_{n+1} + \frac{1+(-1)^n}{n+2}\right ),
\end{equation}
and also
\begin{equation}\label{Hxxx_id3}
\sum_{k=0}^n \binom{n}{k} \frac{(-1)^k}{(k+1)(n-k+1)} H_{k+1} = \frac{(-1)^n}{(n+1)(n+2)}\left ( H_{n+1} + \frac{1+(-1)^n}{n+2}\right ).
\end{equation}
\end{corollary}
\begin{proof}
If we take $z=n$ in \eqref{Har_from_Cx}, all terms are defined except $H_{z-n-1}$. However, we have
\[\binom{z}{n+1}H_{z-n-1}=\frac{z-n}{z+1}\binom{z+1}{n+1}H_{z-n-1}=\frac{\binom{z+1}{n+1}}{z+1}\cdot (z-n)H_{z-n-1}\]
and, using \eqref{Har_psi},
\[\lim_{z\to n} (z-n)H_{z-n-1}=\lim_{z\to n} [(z-n)(\psi(z-n)+\gamma)]=-1\]
where the last equality follows from the fact that $0$ is a simple pole of $\psi(z)$ with residue -1. Thus, 
\[\lim_{z\to n}\binom{z}{n+1}H_{z-n-1}=\frac{-1}{n+1}\]
and \eqref{Hxxx_id1} is deduced by a passage to the limit as $z$ tends to $n$ in \eqref{Har_from_Cx}.\\
The second sum is obtained by setting $z=n+1$ in \eqref{Har_from_Cx}. The third sum follows easily form the second.
\end{proof}

The simple polynomial identity that follows is the source of a new series of results.

\begin{proposition}
The following polynomial identity holds
\begin{equation}\label{eq1}
\sum_{k=0}^n \frac{(x+1)^{k+1}}{k+1} = H_{n+1} + \sum_{k=0}^n \binom{n+1}{k+1} \frac{x^{k+1}}{k+1}.
\end{equation}
\end{proposition}
\begin{proof}
Take $z=-1$ in \eqref{id_Cz}.
\end{proof}

\begin{corollary}
If $n$ is a nonnegative integer, then
\begin{equation}\label{eq 5}
\sum_{k=0}^n \binom{n}{k} \frac{(-1)^k}{(k+1)^2} = \frac{H_{n+1}}{n+1} \qquad\mbox{and}\qquad 
\sum_{k=0}^n \binom{n}{k} \frac{(-1)^k 2^k}{(k+1)^2} = \frac{O_{\lfloor{\frac{n}{2}}\rfloor+1}}{n+1},
\end{equation}
where $O_m$ denotes the $m$th odd harmonic number: $O_m=\sum_{k=1}^m\frac{1}{2k-1}$.
\end{corollary}
\begin{proof}
Take successively $x=-1$ and $x=-2$ in \eqref{eq1}.
\end{proof}

The next corollary extends the first of these results.

\begin{corollary}
If $n$ is a nonnegative integer and $q$ a positive integer, then
\begin{equation}\label{eq 3}
\sum_{k=0}^n \binom{n+1}{k+1} \frac{(-1)^k}{(k+1)(k+q+1)} =(n+1)\sum_{k=0}^n \binom{n}{k} \frac{(-1)^k}{(k+1)^2(k+q+1)} = \frac{H_{n+1}}{q}-\frac{1}{q^2}\left(1-\frac{1}{\binom{n+q+1}{q}}\right)
\end{equation}
and
\begin{equation}\label{eq 4}
\sum_{k=0}^n \binom{n+1}{k+1} \frac{(-1)^k}{k+q+1} = (n+1)\sum_{k=0}^n \binom{n}{k} \frac{(-1)^k}{(k+1)(k+q+1)} 
= \frac{1}{q}\left(1-\frac{1}{\binom{n+q+1}{q}}\right).
\end{equation}
\end{corollary}
\begin{proof}
In \eqref{eq1}, replace $x$ by $-x$ and multiply out by $x^{q-1}$ to obtain
\[\sum_{k=0}^n\frac{(1-x)^{k+1}x^{q-1}}{k+1}=H_{n+1}x^{q-1}-\sum_{k=0}^n\binom{n+1}{k+1}\frac{(-1)^k x^{k+q}}{k+1}\]
Integrating from $0$ to $1$ yields
\[\sum_{k=0}^n\frac{1}{k+1}\cdot\frac{(q-1)!(k+1)!}{(q+k+1)!}=\frac{H_{n+1}}{q}- \sum_{k=0}^n\binom{n+1}{k+1}\frac{(-1)^k }{(k+1)(k+q+1)}\]
The left-hand side rewrites as 
\[(q-1)!\sum_{k=0}^n \frac{1}{(k+1)(k+2)\cdots (k+q+1)}=\frac{(q-1)!}{q}\left(\frac{1}{ q!}-\frac{1}{(n+2)\cdots (n+q+1)}\right),\]
where we have used the fact that
\[\frac{q}{(k+1)(k+2)\cdots (k+q+1)}=\frac{1}{(k+1)\cdots (k+q)}-\frac{1}{(k+2)\cdots (k+q+1)}.\]
We deduce that
\[\frac{(q-1)!}{q}\left(\frac{1}{ q!}-\frac{1}{q!\binom{n+q+1}{q}}\right)=\frac{H_{n+1}}{q}- \sum_{k=0}^n\binom{n+1}{k+1}\frac{(-1)^k }{(k+1)(k+q+1)}\]
and \eqref{eq 3} follows. To obtain \eqref{eq 4}, we use the decomposition 
\[\frac{1}{(k+1)(k+q+1)}=\frac{1}{q}\left(\frac{1}{k+1}-\frac{1}{k+q+1}\right)\]
and \eqref{eq 5}.
\end{proof}

\begin{remark}
By reindexing we see that \eqref{eq1} has the equivalent form
\begin{equation*}
\sum_{k=1}^n \frac{(1+x)^k}{k} = H_n + \sum_{k=1}^n \binom{n}{k} \frac{x^k}{k}.
\end{equation*}
This is Equation (22) in \cite{Adegoke}. 
\end{remark}

Identity \eqref{eq1} generalizes as follows.

\begin{proposition}
If $n,q$ are integers with $n\ge 0$ and $q\ge 1$, then
\begin{align}\label{eq2}
\sum_{k=0}^n \frac{(x+1)^{k+q}}{(k+1)(k+2)\cdots (k+q)} &= \sum_{k=0}^n \binom{n+1}{k+1}\frac{x^{k+q}}{(k+1)(k+2)\cdots (k+q)} \nonumber \\
& + H_{n+1}\frac{x^{q-1}}{(q-1)!} + \frac{1}{(q-1)!}\sum_{j=1}^{q-1}\frac{x^{q-1-j}}{j}\binom{q-1}{j}\left(1-\frac{1}{\binom{n+1+j}{j}}\right).
\end{align}
\end{proposition}
\begin{proof}
The proof is by induction on $q$. The case $q=1$ is \eqref{eq1}. The induction step from $q$ to $q+1$ is obtained by first integrating \eqref{eq2} from $0$ to $X$ (and renaming $X$ as $x$ afterward). In the end, all boils down to showing that 
\[\sum_{k=0}^n\frac{1}{(k+1)(k+2)\cdots (k+q+1)}+\frac{1}{(q-1)!}\sum_{j=1}^{q-1}\frac{x^{q-j}}{j(q-j)}\binom{q-1}{j}\left(1-\frac{1}{\binom{n+1+j}{j}}\right)\]
is equal to
\[\frac{1}{q!}\sum_{j=1}^q\frac{x^{q-j}}{j}\binom{q}{j}\left(1-\frac{1}{\binom{n+1+j}{j}}\right).\]
This follows from 
$$\sum_{k=0}^n\frac{1}{(k+1)(k+2)\cdots (k+q+1)} = \frac{1}{q\cdot q!}\left(1-\frac{1}{\binom{n+1+q}{q}}\right)$$ 
(met earlier) and 
$$\frac{q}{j(q-j)} \binom{q-1}{j} = \frac{1}{j}\binom{q}{j},$$ 
which is readily checked.
\end{proof}

\begin{corollary}
If $n,q$ are integers with $n\ge 0$ and $q\ge 1$, then
\begin{equation}\label{eq 6}
\sum_{k=0}^n\binom{n+1}{k+1}\frac{(-1)^k}{(k+1)\cdots (k+q)}=\frac{1}{(q-1)!}\left(H_{n+1}-\sum_{j=1}^{q-1}\frac{(-1)^{j-1}}{j}\binom{q-1}{j}\left(1-\frac{1}{\binom{n+1+j}{j}}\right)\right).
\end{equation}
\end{corollary}
\begin{proof}
Take $x=-1$ in \eqref{eq2}. Note that for $q\ge 2$, the left-hand side is also
\[(n+1)\sum_{k=0}^n\binom{n}{k}\frac{(-1)^k}{(k+1)^2(k+2)\cdots (k+q)}\]
and that \eqref{eq 6} can also be written as
\[\sum_{k=0}^n \frac{(-1)^k\binom{n+1}{k+1}}{\binom{k+q}{q}} = q\left(H_{n+1}-\sum_{j=1}^{q-1}\frac{(-1)^{j-1}}{j}\binom{q-1}{j}\left(1-\frac{1}{\binom{n+1+j}{j}}\right)\right).\]
\end{proof}

\begin{remark}
Identities involving harmonic numbers of the same kind as those obtained above are not of isolated nature and similar results exist. For instance, 
in Boyadzhiev's book \cite{Boyadzhiev}, Entry (8.38), we find
\begin{equation*}
\sum_{k=0}^n \binom{n}{k} \frac{(-1)^k}{(m+k)^2} = \frac{H_{m+n}-H_{m-1}}{m \binom{m+n}{m}}.
\end{equation*}
Janji\'{c} \cite[Proposition 8]{Janjic} derived the identity
\begin{equation*}
\sum_{k=0}^n \binom{m}{k} \frac{(-1)^k}{n+1-k} = (-1)^n \binom{m}{n+1} \left ( H_{m}-H_{m-n-1}\right ), \quad 0\leq n+1\leq m.
\end{equation*}
Alzer and Richards \cite[Remark 1]{Alzer} offer
\begin{equation*}
\frac{p}{n}+\frac{1}{n^p} \sum_{k=0}^{p-1} \binom{n}{k} \binom{n-k-1}{n-p} (-1)^{p-k-1} \frac{k^p}{n-k} = H_n - H_{n-p}, \quad n\geq p.
\end{equation*}
\end{remark}

For our next series of results, we need the following lemma:

\begin{lemma}\label{beta_int}
For integers $q,k$ such that $q\ge 1,k\ge 0$, we have
$$\int_0^1 x^{q-1} (1-x)^k\ln (x)\,dx = -(H_{q+k}-H_{q-1})\frac{1}{q\binom{q+k}{k}}=-(H_{q+k}-H_{q-1})\int_0^1 x^{q-1}(1-x)^k\,dx$$
\end{lemma}
\begin{proof}
The result
\[\int_0^1 x^{q-1}(1-x)^k\,dx=\frac{(q-1)!k!}{(q+k)!}=\frac{1}{q\binom{q+k}{k}}\]
is well-known and has been used before. The lemma then follows from $\int_0^1 x^{q-1}\ln(x)\,dx=-\frac{1}{q^2}$ and induction on $k$ (alternatively, the result can be found in \cite{GrRy07} as Entry 4.253).   
\end{proof}

\begin{proposition}
If $q$ is a positive integer and $z$ a complex number such that $z\notin\{0,1,\ldots,n\}$, then
\begin{equation}\label{id_Cu}
\sum_{k=0}^n \frac{\binom{z}{k}}{\binom{q+k}{k}}(-1)^k
= (-1)^n q(n+1) \binom{z}{n+1} \sum_{k=0}^n \binom{n}{k} \frac{(-1)^k}{(z-k)(q+k)}
\end{equation}
and
\begin{equation}\label{id_Cv}
\sum_{k=0}^n \frac{\binom{z}{k}}{\binom{q+k}{k}} (-1)^k\left( H_{q+k} - H_{q-1} \right) 
= (-1)^n q(n+1) \binom{z}{n+1} \sum_{k=0}^n \binom{n}{k} \frac{(-1)^k}{(z-k)(q+k)^2}.
\end{equation}
\end{proposition}
\begin{proof}
In \eqref{main_lem_id}, replace $x$ by $-x$ and multiply by $x^{q-1}$. Integration from $0$ to $1$ gives \eqref{id_Cu}. In the same way \eqref{id_Cv} is obtained if we multiply by $x^{q-1}\ln(x)$ and integrate from $0$ to $1$. \\
\end{proof}

\begin{remark}
Since 
$$q\binom{q+k}{k} = (q+k)\binom{q-1+k}{k} = (-1)^k(q+k)\binom{-q}{k},$$
identity \eqref{id_Cu} also writes as 
\begin{equation}\label{id_Cw} 
\sum_{k=0}^n \frac{\binom{z}{k}}{\binom{-q}{k}}\frac{1}{q+k}=\sum_{k=0}^n \frac{\binom{z}{k}}{\binom{q-1+k}{k}}\frac{(-1)^k}{q+k}=
 (-1)^n (n+1) \binom{z}{n+1} \sum_{k=0}^n \binom{n}{k} \frac{(-1)^k}{(z-k)(q+k)}.
\end{equation}
Similar modifications can be made to \eqref{id_Cv}.
\end{remark}

Simplifying further, \eqref{id_Cu} gives the following remarkable evaluation.

\begin{corollary}
Let $q$ be a positive integer and let $z\in\mathbb{C}\setminus\{-q,-(q-1),\ldots,-2,-1,0,1,\ldots,n\}$. Then, we have
\begin{equation}\label{id_Ce}
\sum_{k=0}^n\binom{n}{k}\frac{(-1)^k}{(z-k)(q+k)}=\frac{\binom{z+q-1}{n+q}+(-1)^n\binom{z+q-1}{q-1}}{\binom{z+q}{q} q(n+1)\binom{z}{n+1}}.
\end{equation}
In particular, for $n\geq 0$
$$\sum_{k=0}^n \binom{n}{k} \frac{(-1)^{k}}{(n+1)^2-k^2} = \frac{1+(-1)^n\binom{2n+1}{n}}{2(n+1)^2\binom{2n+1}{n}}$$
and for $n\geq 1$
$$\sum_{k=0}^n \binom{n}{k} \frac{(-1)^{k}}{4n^2-k^2} = \frac{\binom{4n-1}{3n}+(-1)^n\binom{4n-1}{2n-1}}{2 n^2 \binom{2n}{n}\binom{4n}{2n}}.$$
\end{corollary}
\begin{proof}
It suffices to prove that
\begin{equation}\label{id_Cf}
\sum_{k=0}^n(-1)^k\frac{\binom{z}{k}}{\binom{q+k}{k}}=\frac{\binom{z+q-1}{q-1}+(-1)^n\binom{z+q-1}{n+q}}{\binom{z+q}{q}}.
\end{equation}
Now, it is easily checked that 
\[\frac{\binom{z}{k}}{\binom{q+k}{k}} = \frac{\binom{z+q}{q+k}}{\binom{z+q}{q}}.\]
In addition, we have
\begin{align*}
\sum_{k=0}^n (-1)^k \binom{z+q}{q+k} &= (-1)^q\sum_{j=q}^{n+q}(-1)^j\binom{z+q}{j} \\
&= (-1)^q\left((-1)^{n+q}\binom{z+q-1}{n+q}-(-1)^{q-1}\binom{z+q-1}{q-1}\right).
\end{align*}
The result follows. The particular cases are obtained be setting $z=q=n+1$ and $z=q=2n$, respectively.
\end{proof}

\begin{corollary} 
If $n,q$ are integers with $n\ge 0,q\ge 1$, then
\begin{multline*}
\shoveright{\sum_{j=1}^{n+1}(-1)^{j-1}\binom{n+q+1}{j}(H_{n+1}-H_j)= H_{n+1}\left(1+(-1)^n\binom{n+q}{n+1}\right)\qquad\qquad\qquad\qquad}\\%
-q(n+1)\binom{n+q+1}{q}\sum_{k=0}^n\frac{(-1)^k\binom{n}{k}}{(k+q)(n+1-k)^2}.
\end{multline*}
\end{corollary}
\begin{proof}
We differentiate \eqref{id_Cu} with respect to $z$ and use \eqref{id_Cu} itself to obtain 
\begin{align*}
&\sum_{k=0}^n \frac{(-1)^k\binom{z}{k}}{\binom{q+k}{k}}(H_z-H_{z-k}) \\
&\qquad = (H_z-H_{z-n-1})\sum_{k=0}^n\frac{(-1)^k\binom{z}{k}}{\binom{q+k}{k}}+(-1)^{n+1} q(n+1)\binom{z}{n+1}\sum_{k=0}^n\frac{(-1)^k\binom{n}{k}}{(k+q)(z-k)^2}.
\end{align*}
Taking $z=n+1$ and using \eqref{id_Cf} lead to 
\begin{align*}
 \sum_{k=0}^n(-1)^k\binom{n+q+1}{q+k}(H_{n+1}-H_{n+1-k}) &= H_{n+1}\left((-1)^n+\binom{n+q}{n+1}\right) \\
 & +(-1)^{n+1} q(n+1)\binom{n+q+1}{q}\sum_{k=0}^n\frac{(-1)^k\binom{n}{k}}{(k+q)(n+1-k)^2}
 \end{align*}
 and a change of index of summation in the first sum gives the result.
\end{proof}

\begin{remark} After some easy calculations, the particular case $q=1$ writes as 
\begin{equation}\label{eq_12}  
\sum_{j=1}^{n+2}(-1)^{j-1}\binom{n+2}{j}H_j=\frac{1}{n+2},
\end{equation}
a known result.
\end{remark}
The last identity of the section uses \eqref{id_Cv}.
\begin{corollary}
For each $n\geq 0$ we have
\begin{equation*}
\sum_{k=0}^n \binom{n}{k} \frac{(-1)^{k+1}}{(2k-1)(k+1)^2} 
= \frac{2^{2n+2}}{9\binom{2n}{n}} + \frac{1}{3(n+1)}\left ( H_{n+1} + \frac{2}{3}\right ).
\end{equation*}
\end{corollary}
\begin{proof}
Set $q=1$ and $z=1/2$ in \eqref{id_Cv}. This yields
$$\sum_{k=0}^n \frac{2^{-2k}}{2k-1} C_k H_{k+1} = (n+1)C_n 2^{-2n}\sum_{k=0}^n \binom{n}{k} \frac{(-1)^{k}}{(2k-1)(k+1)^2},$$
where $C_n=1/(n+1)\binom{2n}{n}$ denotes the $n$the Catalan number. But Bataille and Frontczak \cite[Corollary 4.2]{Bataille2} 
have shown that
$$\sum_{k=1}^n \frac{2^{-2k}}{2k-1} C_k H_{k+1} = \frac{5}{9} - \frac{2^{-2n}}{3} C_n \left (H_{n+1} + \frac{2}{3} \right ).$$
This instantly shows that
$$\sum_{k=0}^n \frac{2^{-2k}}{2k-1} C_k H_{k+1} = - \frac{4}{9} - \frac{2^{-2n}}{3} C_n \left (H_{n+1} + \frac{2}{3} \right )$$
and the statement follows.
\end{proof}

\section{A particular sum}

This section is devoted to the study of the sums
$$S_n(q) = \sum_{k=0}^n \binom{n}{k} \frac{(-1)^{k}}{(k+q)(n-k+q)}, \qquad q\geq 1,$$
and to some related problems. We start with a closed form for the sum, providing two proofs.

\begin{proposition}
Let $q$ be a positive integer. Then we have
\begin{equation}\label{id_C6}
S_n(q) = \frac{(-1)^n + 1}{q (n+2q)\binom{n+q}{n}}.
\end{equation}
\end{proposition}
\begin{proof} First Proof: use  \eqref{id_Ce} with $z=n+q$.\\
Second Proof: As we have the decomposition
$$\frac{1}{(k+q)(n-k+q)} = \frac{1}{n+2q}\left ( \frac{1}{k+q} + \frac{1}{n-k+q}\right ),$$
we make use of \eqref{eq_9} and get
\begin{align*}
S_n(q) &= \frac{1}{n+2q}\left (\sum_{k=0}^n \binom{n}{k} \frac{(-1)^{k}}{k+q} + \sum_{k=0}^n \binom{n}{k} \frac{(-1)^{k}}{n-k+q} \right ) \\
&= \frac{1}{n+2q}\left (\sum_{k=0}^n \binom{n}{k} \frac{(-1)^{k}}{k+q} + (-1)^n \sum_{k=0}^n \binom{n}{k} \frac{(-1)^{k}}{k+q} \right ) \\
&= \frac{1}{n+2q}\left ( \frac{1}{q\binom{n+q}{n}} + (-1)^n \frac{1}{q\binom{n+q}{n}}\right ).
\end{align*} 
\end{proof}

\begin{corollary}
\begin{equation*}
S_n(1) = 
\begin{cases}
0, & \text{\rm if $n$ is odd;} \\ 
\frac{2}{(n+1)(n+2)}, & \text{\rm if $n$ is even;} 
\end{cases}
\end{equation*}
and
\begin{equation*}
S_n(2) = 
\begin{cases}
0, & \text{\rm if $n$ is odd;} \\ 
\frac{2}{(n+1)(n+2)(n+4)}, & \text{\rm if $n$ is even.} 
\end{cases}
\end{equation*}
\end{corollary}

A natural generalization of $S_n(q)$ is the sum
$$S_n(q;x) = \sum_{k=0}^n \binom{n}{k} \frac{(-1)^{k} x^k}{(k+q)(n+q-k)}.$$ 
The initial formula \eqref{main_lem_id} with $z=n+q$ and $-x$ instead of $x$ gives
\[\sum_{k=0}^n \binom{n+q}{k} (-1)^k (1-x)^k = K_{n,q}\sum_{k=0}^n \binom{n}{k} \frac{(-1)^k x^k}{n+q-k}
= K_{n,q}\sum_{k=0}^n \binom{n}{k}\frac{(-1)^{n-k} x^{n-k}}{k+q},\]
where $K_{n,q}=(-1)^n(n+1)\binom{n+q}{n+1}$. In order to replace $x^{n-k}$ by $x^k$ in the rightmost sum, we change $x$ into $\frac{1}{x}$. This yields
\[K_{n,q}\sum_{k=0}^n\binom{n}{k}\frac{(-1)^{k} x^{k}}{k+q} = (-1)^n\sum_{k=0}^n \binom{n+q}{k}x^{n-k}(1-x)^k.\]
It follows that 
\begin{equation}\label{xq_id}
(n+2q)K_{n,q}\sum_{k=0}^n \binom{n}{k} \frac{(-1)^{k}x^k}{(k+q)(n+q-k)}=\sum_{k=0}^n\binom{n+q}{k}(-1)^k(1-x)^k(1+(-1)^{n-k}x^{n-k}),
\end{equation}
which, with $x=1$, gives \eqref{id_C6} again and also leads to the following result.

\begin{proposition}
Let $q$ be a positive integer and
\[U(n,q) = \sum_{k=0}^n \binom{n}{k} \frac{(-1)^{k}}{(k+1)(k+q)(n+q-k)} .\]
Then we have 
\begin{equation}\label{Un1_id}
U(n,1) = \sum_{k=0}^n \binom{n}{k} \frac{(-1)^{k}}{(k+1)^2(n+1-k)} = \frac{(n+2)H_{n+2}+(-1)^n}{(n+1)(n+2)^2},
\end{equation}
and if $q\ge 2$
\begin{equation}
U(n,q) = \frac{1}{(n+1)(q-1)(n+q+1)}+\frac{(-1)^n}{q(n+q+1)(n+2q)\binom{n+q}{n}}-\frac{1}{q(q-1)(n+2q)\binom{n+q}{n}}.
\end{equation}
\end{proposition}
\begin{proof}
Integrate \eqref{xq_id} from $0$ to $1$, giving
\begin{equation}\label{eq_11}
\sum_{k=0}^n \frac{(-1)^{k}\binom{n}{k}}{(k+1)(k+q)(n+q-k)} = \frac{(-1)^n}{q(n+2q) \binom{n+q}{n}}\left(\frac{(-1)^n}{n+1}\sum_{k=0}^n\frac{\binom{n+q}{k}}{\binom{n}{k}}+\sum_{k=0}^n (-1)^k \frac{\binom{n+q}{k}}{k+1}\right).
\end{equation}
If $q=1$, we obtain
\[U(n,1) = \frac{(-1)^n}{(n+2)(n+1)}\left((-1)^n H_{n+1}+\sum_{k=0}^n\binom{n+1}{k}\frac{(-1)^k}{k+1}\right).\]
Since
\[\sum_{k=0}^n\binom{n+1}{k}\frac{(-1)^k}{k+1} = \frac{1}{n+2}\sum_{k=0}^n\binom{n+2}{k+1}(-1)^k=\frac{1+(-1)^n}{n+2}\]
a short calculation gives the expected result. \\[0.2cm]
If $q\ge 2$, we first use the formula \cite[Identity (A)]{Bataille2}
\[\sum_{k=0}^n\binom{m+k}{k} x^k=\sum_{k=0}^n\binom{m+n+1}{k}x^k(1-x)^{n-k}.\]
Integrating from 0 to 1 yields
\[\sum_{k=0}^n\frac{\binom{m+n+1}{k}}{\binom{n}{k}}=(n+1)\sum_{k=0}^n\binom{m+k}{k}\frac{1}{k+1}.\]
With $m=q-1$ and the help of \eqref{eq_8}, we are led to 
\[\sum_{k=0}^n\frac{\binom{n+q}{k}}{\binom{n}{k}}=(n+1)\sum_{k=0}^n\binom{q-1+k}{k}\frac{1}{k+1}=\frac{n+1}{q-1}\left(\binom{n+q}{n+1}-1\right).\]
Second, from \eqref{id_Cx}, we deduce
\[\sum_{k=0}^n (-1)^k \frac{\binom{n+q}{k}}{k+1}=\frac{1}{n+q+1}\left(1+(-1)^{n}\binom{n+q}{n+1}\right).\]
Returning to \eqref{eq_11}, a simple calculation gives the result.
\end{proof}

As an example, here is the case $q=2$:
\begin{corollary}
We have
\begin{equation}\label{Un2_id}
\sum_{k=0}^n \binom{n}{k} \frac{(-1)^{k}}{(k+1)(k+2)(n+2-k)} = \frac{n^2+5n+5+(-1)^n}{(n+1)(n+2)(n+3)(n+4)}.
\end{equation}
\end{corollary}

We also have the following evaluations.

\begin{corollary}
For $n\geq 0$,
\begin{equation}
\sum_{k=0}^n \binom{n}{k} \frac{(-1)^{k}}{(k+1)^2(k+2)(n+1-k)} = \frac{(n+2)(n+3)H_{n+2} + (-1)^n - (n+2)^2}{(n+1)(n+2)^2(n+3)},
\end{equation}
\begin{equation}
\sum_{k=0}^n \binom{n}{k} \frac{(-1)^{k}}{(k+1)^2(k+2)(n+2-k)} 
= \frac{(n+2)(n+3)(n+4)H_{n+2} + (-1)^n - n^3-8n^2-21n-19 }{(n+1)(n+2)(n+3)^2(n+4)},
\end{equation}
and
\begin{equation}
\sum_{k=0}^n \binom{n}{k} \frac{(-1)^{k}}{(k+1)(k+2)^2(n+2-k)} 
= \frac{n^3+10n^2+31n+29 - (n+3)(n+4)H_{n+2} + (-1)^n}{(n+1)(n+2)(n+3)(n+4)^2}.
\end{equation}
\end{corollary}
\begin{proof}
Combine equations \eqref{Un1_id} and \eqref{Un2_id} with the following identities which can be found in 
\cite[Eqs. (4.20) and (4.(21)]{Adegoke25}
$$\sum_{k=0}^n \binom{n}{k} \frac{(-1)^{k}}{(k+1)^2(k+2)} =\frac{H_{n+2}-1}{n+1}$$
and
$$\sum_{k=0}^n \binom{n}{k} \frac{(-1)^{k}}{(k+1)(k+2)^2} =\frac{n+2-H_{n+2}}{(n+1)(n+2)}.$$
\end{proof}

\section{Concluding Comments}

This paper is based on a fundamental lemma involving two complex parameters $z$ and $x$. This result turned out to be extremely productive,
allowing us to present a range of new (polynomial) combinatorial identities. Some of these generalize existing relations, others produce new results or rediscover known material. It is obvious, however, that the findings derived in this paper can be exploited further. For example, a considerable amount of new relations for second-order recurrent sequences can be obtained. To keep the paper readable, we just indicate a few Fibonacci relations that are consequences of specific results from the previous sections. \\

We recall the basic facts: The Fibonacci numbers are defined by $F_0=0,F_1=1$ and the recursion $F_{n+2}=F_{n+1}+F_n$ for $n\in\mathbb{Z}$. The Lucas numbers $L_n$ satisfy the same recursion but begin with $L_0=2$ and $L_1=1$. We know that 
$$F_n=\frac{\alpha^n-\beta^n}{\sqrt{5}} \qquad\text{and}\qquad L_n=\alpha^n+\beta^n,$$
where $\alpha=\frac{1+\sqrt{5}}{2}$ and $\beta=\frac{1-\sqrt{5}}{2}$ are the roots of the equation $x^2-x-1=0$.\\
A number of useful relations are satisfied by the powers of $\alpha$ (and $\beta$), e.g. $1+\alpha=\alpha^2,1+\alpha^4=3\alpha^2,1+4\alpha^3=\alpha^6...$. More generally, using the Lucas numbers $L_r$ we have 
\[1+(-1)^{r-1}\alpha^r L_r=(-1)^{r-1}\alpha^{2r}\quad (\mbox{or}\  1+(-1)^r\alpha^{2r}=(-1)^r\alpha^r L_r )\]
and similar relations with $\beta$ replacing $\alpha$. \\

In the next propositions, we use \eqref{id_Ca} to obtain interesting identities. In \eqref{id_Ca}, we first take $x=(-1)^{r-1}\alpha^r L_r$, then $x=(-1)^{r-1}\beta^r L_r$. Combining appropriately, we obtain that for integers $r,n$ with $n\ge 0$:  
\begin{proposition}
\begin{align*}
(-1)^{n(r-1)}\sum_{k=0}^n\frac{\binom{n}{k}}{\binom{2k}{k}}2^{2k}(-1)^{k+1}L_r^k F_{2rn-k}&=\sum_{k=0}^n\binom{n}{k}(-1)^{k(r-1)}\frac{L_r^k F_{rk}}{2k-1}\\
&=\frac{2^{2n}}{\binom{2n}{n}}\sum_{k=0}^n\binom{2k}{k}(-1)^{k(r-1)}\frac{2^{-2k} F_{2rk}}{2k-1}
\end{align*}
\end{proposition}
Similarly, with $x=(-1)^r\alpha^{2r}$ (and $x=(-1)^r\beta^{2r}$) we get
\begin{proposition}
\begin{align*}
(-1)^{rn}\sum_{k=0}^n\frac{\binom{n}{k}}{\binom{2k}{k}}2^{2k}(-1)^{k+1}L_r^{n-k} F_{r(n+k)}&=\sum_{k=0}^n\binom{n}{k}(-1)^{rk}\frac{F_{2rk}}{2k-1}\\
&=\frac{2^{2n}}{\binom{2n}{n}}\sum_{k=0}^n\binom{2k}{k}(-1)^{rk)}\frac{2^{-2k} F_{rk}L_r^k}{2k-1}
\end{align*}
\end{proposition}

\begin{remark}
Of course, there exist other useful relations between the powers of $\alpha$ (and $\beta$). The reader  will easily check that the examples $1+\alpha^3=2\alpha^2$ and $ 1+2\alpha=\alpha^3$ provide
\[\sum_{k=0}^n\frac{\binom{n}{k}}{\binom{2k}{k}}(-1)^{k+1}2^{n+k}F_{2n+k}=\sum_{k=0}^n\binom{n}{k}\frac{F_{3k}}{2k-1}=\frac{2^{2n}}{\binom{2n}{n}}\sum_{k=0}^n\binom{2k}{k}\frac{F_{2k}}{2^k(2k-1)}\]
and
\[\sum_{k=0}^n\frac{\binom{n}{k}}{\binom{2k}{k}}(-1)^{k+1}2^{3k}F_{3n-2k}=\sum_{k=0}^n\binom{n}{k}\frac{2^k F_{k}}{2k-1}=\frac{2^{2n}}{\binom{2n}{n}}\sum_{k=0}^n\binom{2k}{k}\frac{F_{3k}}{2^{2k}(2k-1)}.\]
\end{remark}

To conclude this study, here is a surprising identity deduced from the relations obtained in Corollary 5.2. For integers $n,q\ge 1$, we have
\begin{equation*}\label{fib_ex}
\left(\sum_{k=0}^n\frac{\binom{n}{k}F_k}{k+q}\right)\left(\sum_{k=0}^n\binom{n+q}{k}(-1)^k F_k\right)+\left(\sum_{k=0}^n\frac{\binom{n}{k}(-1)^k F_{k}}{n+q-k}\right)\left(\sum_{k=0}^n\binom{n+q}{k}(-1)^k F_{n+k}\right)=0.
\end{equation*}
We leave the proof to the interested readers as a little exercise.

%\section{Conclusion}

%This paper is based on an polynomial identity involving two complex parameters. This identity turned out to be extremely rich
%and allowed us to explore a broad field of applications. Starting with a generalization of the hockey-stick identity we have given 
%a range of surprising and remarkable combinatorial identities.   

%\section{Acknowledgments}
%The authors would like to express their gratitude to the anonymous referees.

\end{document}